\documentclass{article}

\usepackage{arxiv}

\usepackage[utf8]{inputenc} 
\usepackage[T1]{fontenc}    
\usepackage{hyperref}       
\usepackage{url}            
\usepackage{booktabs}       
\usepackage{amsfonts}       
\usepackage{nicefrac}       
\usepackage{microtype}      
\usepackage{lipsum}
\usepackage{graphicx}
\usepackage{tabulary,xcolor}
\usepackage{amsfonts,amsmath,amssymb}
\usepackage{mathtools}
\usepackage{commath}
\usepackage[T1]{fontenc}
\usepackage{enumitem}
\usepackage{amsthm}
\usepackage{ifluatex}
\ifluatex
\usepackage{fontspec}
\defaultfontfeatures{Ligatures=TeX}
\usepackage[]{unicode-math}
\unimathsetup{math-style=TeX}
\else
\usepackage[utf8]{inputenc}
\fi
\ifluatex\else\usepackage{stmaryrd}\fi

\usepackage{url,multirow,morefloats,floatflt,cancel,tfrupee}

\usepackage{colortbl}
\usepackage{xcolor}
\usepackage{pifont}
\usepackage[nointegrals]{wasysym}
\makeatletter

\def\mcWidth#1{\csname TY@F#1\endcsname+\tabcolsep}

\def\cAlignHack{\rightskip\@flushglue\leftskip\@flushglue\parindent\z@\parfillskip\z@skip}
\def\rAlignHack{\rightskip\z@skip\leftskip\@flushglue \parindent\z@\parfillskip\z@skip}
\if@twocolumn\usepackage{dblfloatfix}\fi
\AtBeginDocument{
\expandafter\ifx\csname eqalign\endcsname\relax
\def\eqalign#1{\null\vcenter{\def\\{\cr}\openup\jot\m@th
  \ialign{\strut$\displaystyle{##}$\hfil&$\displaystyle{{}##}$\hfil
      \crcr#1\crcr}}\,}
\fi
}
\AtBeginDocument{%
  \@ifpackageloaded{endfloat}%
   {\renewcommand\efloat@iwrite[1]{\immediate\expandafter\protected@write\csname efloat@post#1\endcsname{}}}{}%
}%

\let\lt=<
\let\gt=>
\def\processVert{\ifmmode|\else\textbar\fi}

\@ifundefined{subparagraph}{
\def\subparagraph{\@startsection{paragraph}{5}{2\parindent}{0ex plus 0.1ex minus 0.1ex}%
{0ex}{\normalfont\small\itshape}}%
}{}
\newcommand\role[1]{\unskip}
\newcommand\aucollab[1]{\unskip}
\@ifundefined{tsGraphicsScaleX}{\gdef\tsGraphicsScaleX{1}}{}
\@ifundefined{tsGraphicsScaleY}{\gdef\tsGraphicsScaleY{.9}}{}
\def\checkGraphicsWidth{\ifdim\Gin@nat@width>\linewidth
	\tsGraphicsScaleX\linewidth\else\Gin@nat@width\fi}

\def\checkGraphicsHeight{\ifdim\Gin@nat@height>.9\textheight
	\tsGraphicsScaleY\textheight\else\Gin@nat@height\fi}

\def\fixFloatSize#1{}
\let\ts@includegraphics\includegraphics

\def\inlinegraphic[#1]#2{{\edef\@tempa{#1}\edef\baseline@shift{\ifx\@tempa\@empty0\else#1\fi}\edef\tempZ{\the\numexpr(\numexpr(\baseline@shift*\f@size/100))}\protect\raisebox{\tempZ pt}{\ts@includegraphics{#2}}}}

\AtBeginDocument{\def\includegraphics{\@ifnextchar[{\ts@includegraphics}{\ts@includegraphics[width=\checkGraphicsWidth,height=\checkGraphicsHeight,keepaspectratio]}}}

\def\URL#1#2{\@ifundefined{href}{#2}{\href{#1}{#2}}}

\def\UrlOrds{\do\*\do\-\do\~\do\'\do\"\do\-}%
\g@addto@macro{\UrlBreaks}{\UrlOrds}
\makeatother

\usepackage{multirow}
\usepackage{amssymb}
\usepackage{pifont}
\usepackage{mathrsfs}
\usepackage{amsfonts,amssymb,graphicx,listings,txfonts,ulem}
\usepackage{algorithm}
\usepackage{algpseudocode}
\usepackage{fancyhdr}
\usepackage{subfigure}
\usepackage{soul,color,xcolor,bm}
\usepackage{epstopdf}
\usepackage[toc,page]{appendix}

\newtheorem{theorem}{Theorem}[section]
\newtheorem{lemma}[theorem]{Lemma}

\title{A GREEDY ALGORITHM FOR SPARSE PRECISION MATRIX APPROXIMATION}

\author{
  Didi Lv \\
  School of Mathematical Sciences\\
  Shanghai Jiao Tong University\\
  Shanghai, China, 200240 \\
  \texttt{Eric2014\_Lv@sjtu.edu.cn} \\
   \And
 Xiaoqun Zhang \\
  School of Mathematical Sciences \\ Institute of Natural Sciences and MOE-LSC\\
  Shanghai Jiao Tong University\\
  Shanghai, China, 200240 \\
  \texttt{xqzhang@sjtu.edu.cn} \\
}

\begin{document}%
\maketitle

\begin{abstract}
Precision matrix estimation is an important problem in statistical data analysis.  This paper introduces a fast sparse precision matrix estimation algorithm, namely GISS$^{{\rho}}$, which is originally introduced for compressive sensing. The algorithm GISS$^{{\rho}}$ is derived based on $l_1$ minimization while with the computation advantage of greedy algorithms. We analyze the asymptotic convergence rate of the proposed GISS$^{{\rho}}$  for sparse precision matrix estimation and sparsity recovery properties with respect to the stopping criteria. Finally, we numerically compare GISS$^{\rho}$ to other sparse recovery algorithms, such as  ADMM and  HTP in three settings of precision matrix estimation. The numerical results show the advantages of the proposed algorithm.
\end{abstract}

\keywords{Precision matrix estimation, CLIME estimator, sparse recovery, inverse scale space method, greedy methods.}

\section{Introduction}
Covariance matrix and precision matrix estimation are two important problems in statistics analysis and data science. The problems become more challenging in high-dimensional setting where the number of variable dimension $p$ is larger than the sample size $n$, hence the need for a fast, accurate and stable precision/covariance matrix estimation is  necessary. In the high-dimensional setting, classical methods and theoretical  results with fixed $p$ and large $n$ are no longer applicable.  Another huge challenge due to  high dimension is  high computational cost. Therefore, effective model and method are urgent facing high-dimensional data challenge. \par

Denote $\bm{X}=(X_1,X_2,\cdots,X_p)^T$ by a $p$ variate random vector. The covariance matrix and precision matrix can be traditionally denoted by $\Sigma_0$ and $\Omega_0=\Sigma^{-1}_0$ respectively. Assume an independent and identically distributed $n$ random samples $\{\bm{X}_1,\bm{X}_2, \cdots,\bm{X}_n\}$ are from the distribution of $\bm{X}$. The sample covariance matrix is the common method among the estimators of covariance matrix, which is  defined as follows,
\begin{eqnarray*}
  \Sigma_n=\frac{1}{n-1}\sum_{k=1}^n(\bm{X}_k-\bar{\bm{X}})(\bm{X}_k-\bar{\bm{X}})^T
\end{eqnarray*}
where $\bar{\bm{X}}=\frac{1}{n}\sum_{k=1}^n\bm{X}_k$ denotes the sample mean. When  $p$ is larger than $n$, it is obvious that $\Sigma_n$ is singular and the estimation for $\Omega_0$ naturally becomes unstable and imprecise. \par

Estimation of precision matrix in high-dimensional setting has been studied for a long time. For example, when the random variable  $\bm{X}$ follows a certain ordering structure, methods based on banding the Cholesky factor of the inverse {of sample covariance matrix} were studied in \cite{wu2003nonparametric,bickel2008covariance}. Penalized likelihood methods such as $l_1$-MLE type estimators were studied in \cite{friedman2008sparse,d2008first,yuan2007model} and the convergence rate in Frobenius norm was given by \cite{rothman2008sparse}. In \cite{yuan2010high}, the authors established the convergence rate for sub-Gaussian distribution cases. For more restrictive conditions, such as mutual incoherence or irrepresentable conditions, \cite{ravikumar2011high} showed the convergence rates in elementwise $l_\infty$ norm and spectral norm. To overcome the drawbacks that $l_1$ penalty  inevitably leads to biased estimation, nonconvex penalty such as SCAD penalty \cite{lam2009sparsistency,fan2009network} was proposed {\cite{fan2001variable,zou2006adaptive}}, although it often requires high computational cost. \par

Recently, \cite{CLIME} proposed a new constrained $l_1$ minimization approach called CLIME for sparse precision estimation. Convergence rates in spectral norm,  elementwise $l_\infty$ norm and Frobenius norm were established under weaker assumptions and shown to be faster than those $l_1$-MLE estimators when the population distributions have polynomial-type tails. In addition, CLIME provides an effective computational performance as the columns of precision matrix estimation can be independently obtained and accelerated by parallelized computing. {However, in \cite{CLIME}, each column is obtained by a solving a linear programming, which could be still time-consuming for high dimension}. More efficient approach are still needed to be developed for practical high dimensional applications.\par

In compressive sensing and sparse optimization community, many algorithms and related theoretical results are developed  for  $l_1$ minimization optimization problems \cite{beck2009fast,caiconvergence,cai2009linearized,hale2008fixed,osher2011fast,yin2008bregman,zhang2011unified}. Greedy inverse scale space flows (GISS) \cite{Mller2016FastSR}, originally stems from the adaptive inverse scale (aISS) method \cite{burger2013adaptive}, is a new sparse recovery approach combining the idea of greedy approach and $\ell_1$ minimization. The advantage of GISS method was efficiency compared to the aISS method. GISS$^{\rho}$ with $\rho$ being an acceleration factor, as a variant of GISS, can further accelerate sparse solution recovery by increasing  the support of the current iterate by many indices at once. 
\par

In this article, we take the advantages of CLIME estimator framework and GISS$^{\rho}$ algorithm and propose a new efficient approach for sparse  precision matrix estimation. More specifically, we transfer the the constraint bound in CLIME estimator  to a turning parameter in the stop criteria in GISS$^{\rho}$ method. Convergence result in  elementwise $l_\infty$ norm is established under weaker assumptions same as \cite{CLIME}. Moreover, under the Gaussian noise setting, the stopping time based on the data term is also obtained. The numerical experiments shows the competitive advantage of computation time, {on obtaining the same level of sparsity and accuracy compared to other existing methods.}

The rest of the paper is organized as follows. In Section \ref{sec:pre}, we firstly introduce basic notations  and simply revisit the CLIME estimator. In Section \ref{sec:alg}, we present the basic idea of our new method derived from CLIME estimator and GISS$^{\rho}$ algorithm. In Section \ref{sec:theo}, we establish the theoretical analysis with assumptions. Section \ref{sec:ne} presents the numerical results including simulated experiments and application on real data. We provide the discussion and conclusion in Section \ref{sec:dc} and the proof of the main results can be found in Appendix.
\section{Preliminary}
\label{sec:pre}
\subsection{Notations and definitions}
Before presenting our proposed precision matrix estimator, we first  introduce some essential notations and definitions. \par
For a vector $\bm{a}=(a_1, a_2, \ldots, a_n)^T\in \mathbb{R}^{n}$, we define $|\bm{a}|_1 = \sum_{j=1}^n|a_j|$, $|\bm{a}|_2 = \sqrt{\sum_{j=1}^n a_j^2}$ and $|\bm{a}|_\infty=\mathrm{max}_{1\leq j\leq n}~|a_j|$ to denote the vector norms. Define the two types of inner products between two vectors by $\langle \bm{a},\bm{b}\rangle = \bm{a}^T\bm{b}$ and $\langle\bm{a},\bm{b}\rangle_n=\frac{1}{n}\bm{a}^T\bm{b}$, and denote $\|\bm{\bm{a}}\|_n=\frac{1}{\sqrt{n}}|\bm{a}|_2$. \par

For a matrix $\bm{A} = (a_{ij})\in \mathbb{R}^{p\times p}$, we define the elementwise $l_\infty$ norm $|\bm{A}|_\infty = \mathrm{max}_{1\leq i\leq p, 1\leq j\leq p}|a_{ij}|$, the spectral norm $\|\bm{A}\|_2=\mathrm{sup}_{|\bm{x}|_2\leq1}|\bm{Ax}|_2$, the $L_1$ norm $\|\bm{A}\|_{L_1}=\mathrm{max}_{1\leq j\leq p}\sum_{i=1}^p|a_{ij}|$, the infinity norm $\|\bm{A}\|_\infty=\mathrm{max}_{1\leq i\leq p}\sum_{j=1}^p|a_{ij}|$, the Frobenius norm $\|\bm{A}\|_F=\sqrt{\sum_{i=1}^p\sum_{j=1}^p a_{ij}^2}$, and the elementwise $L_1$ norm $\|\bm{A}\|_1 = \sum_{i=1}^p\sum_{j=1}^p|a_{ij}|$. \par

The $i$th row and $j$th column of the matrix $\bm{A}$ are denoted by $\bm{A}_{i,\cdot}$ and $\bm{A}_{\cdot, j}$ respectively and the transpose of $\bm{A}$ is denoted by $\bm{A}^T$. Let $\bm{A}^*=\frac{1}{p}A^T$ denote the adjoint operator of $\bm{A}$ with respect to the inner product $\langle\cdot,\cdot\rangle_p$. The inner product between two matrix $\bm{A}$ and $\bm{B}$ is denoted by $\langle \bm{A},\bm{B}\rangle=\sum_i(\bm{A}^T\bm{B})_{ii}$ with proper size. For two index sets $T$ and $T'$, we use $\bm{A}_{TT'}$ to denote the $|T|\times|T'|$ matrix with rows and columns of $\bm{A}$ indexed by $T$ and $T'$ respectively and $\bm{A}\succ 0$ means that $\bm{A}$ is positive definite. It also should be noted the $p\times p$ identity matrix is defined by $I$ while $\mathbb{I}$ denoted the indicator function. \par

For two real sequences $\{\epsilon_n\}$ and $\{\eta_n\}$, we write $\epsilon_n=O(\eta_n)$ if there exists a constant $C$ such that $|\epsilon_n|\leq C|\eta_n|$ for large $n$, $\epsilon_n=o(\eta_n)$ if $\mathrm{lim}_{n\rightarrow\infty}{\epsilon_n}/{\eta_n}=0$.
\subsection{CLIME estimator}
The CLIME estimator in~\cite{CLIME} proposed to obtain a precision matrix estimation  via solving the following $l_1$ minimization:
\begin{equation}\label{model1}
  \begin{aligned}
    &\text{min}
    &&\|{\Omega}\|_1 \\
    &\text{s.t}
    &&|\Sigma_n{\Omega} - I|_\infty \leq\lambda_n, \; {\Omega}\in \mathbb{R}^{p\times p}
  \end{aligned}
\end{equation}
where $\Sigma_n\in\mathbb{R}^{p\times p}$ is the sample covariance matrix generated by $n$ data samples, $\Omega$ is  the precision matrix to be estimated and the turning parameter $\lambda_n$ is set for controlling the approximation error under the elementwise $l_\infty$ norm.\par

In general, the solution $\hat{\Omega}_1=(\omega_{ij})$ of (\ref{model1}) is asymmetric, hence a symmetry strategy  as followed was adopted in \cite{CLIME},
\begin{equation}\label{sym}
\begin{aligned}
  &\hat{\Omega} = (\hat{\omega}_{ij}), \;\text{where} \\
  &\hat{\omega}_{ij} = \hat{\omega}_{ji} = {\omega}_{ij}~\mathbb{I}(|\omega_{ij}|\leq|\omega_{ji}|) + {\omega}_{ji}~\mathbb{I}(|\omega_{ij}|>|\omega_{ji}|) \; \text{for} \;i,j=1,2,\cdots, p.
\end{aligned}
\end{equation}
The above defined $\hat{\Omega}$ is the final estimated precision matrix through CLIME estimator. \par

It is easy to see that the convex program (\ref{model1}) can be decomposed into the following $p$ vector convex minimization problems:
\begin{equation}\label{model2}
  \begin{aligned}
    & \text{min}
    && |\bm{\beta_i}|_1\\
    & \text{s.t.}
    && |\Sigma_n\bm{\beta_i} - \bm{e_i}|_\infty \leq\lambda_n, \; \beta_i\in \mathbb{R}^{p},
  \end{aligned}
\end{equation}
for all $1\leq i\leq p$, where $\bm{e_i}$ is a standard unit vector in $\mathbb{R}^{p}$. Denote $\hat{\bm{\beta_i}}$  as the solution of (\ref{model2}), then the assembling  of the $p$  vectors in the form of $[\hat{\bm{\beta}}_1,\hat{\bm{\beta}}_2,\cdots,\hat{\bm{\beta}}_p]$ is a solution of (\ref{model1}), and vice visa.

For the case that the covariance matrix $\Sigma_n$ is not invertible, we can also consider a regularized
\begin{align}
\Sigma_{n,\gamma}&=\Sigma_n+\gamma I
\label{eq:sigma_regu}
\end{align} for $\gamma>0$. Our GISS$^\rho$ estimator can be obtained via solving the following minimization problem:
\begin{equation} \label{model3}
\begin{aligned}
    &\text{min}
    &&|\bm{\beta_i}|_1\\
    &\text{s.t.}
    &&\Sigma_{n,\gamma}\bm{\beta_i} = \bm{e_i}, \; \bm{\beta_i}\in \mathbb{R}^{p}
\end{aligned}
\end{equation}
for all $1\leq i \leq p$ with the $l_\infty$ norm stopping criterion $i.e.$ $|\Sigma_{n,\gamma}\bm{\beta_i} -\bm{e_i}|_\infty\leq\lambda_n$. \par%
\section{Sparse precision matrix estimation by GISS$^{\rho}$}
\label{sec:alg}
Before presenting the GISS$^\rho$ estimator for sparse precision matrix, we first  introduce the basic ideas of GISS and GISS$^\rho$. It stems from the adaptive inverse scale (aISS) method  \cite{burger2013adaptive}, which was firstly proposed for solving compressive sensing problem. In the following, we present the algorithms in the context of  precision matrix estimation for the consistency of notations.

\subsection{Inverse scale space (ISS) based algorithms}

 The aISS algorithm was proposed to  solve the equality constrained $l_1$ minimization (basic pursuit) problem. In the context of precision matrix estimation, we consider
\begin{equation}\label{aISS_prob}
\begin{aligned}
  & \text{min}
  && |\bm{\beta}|_1\\
  & \text{s.t.}
  && \Sigma\bm{\beta}=\bm{e}
\end{aligned}
\end{equation}
where $\Sigma\in \mathbb{R}^{p\times p}$ denotes the covariance matrix, $\bm{\beta}\in \mathbb{R}^{p}$ represents one column of the sparse precision matrix to be reconstructed and $\bm{e}\in \mathbb{R}^{p}$ denotes one column of identity matrix $I$. \par
For the  $l_1$ minimization optimization problem for a general matrix $\Sigma$ and vector $\bm{e}$, Bregman iteration (BI) and augmented Lagrangian (AL) method are two efficient equivalent iterative methods~\cite{osher2005iterative,esser2009applications,yin2008bregman}.
By applying  BI  to solve (\ref{aISS_prob}), we obtain the following iterative sequence:
\begin{subequations}\label{AL_seq_2}
\begin{align}
&\bm{\beta}^{k+1} = \mathrm{arg\ min}\ |\bm{\beta}|_1 + \frac{\lambda}{2} |\Sigma\bm{\beta}-\bm{e}|_2^2-\langle \bm{p}^k,\bm{\beta}\rangle\label{primal},\\
&\bm{p}^{k+1} = \bm{p}^k + \lambda \Sigma^T(\bm{e}-\Sigma \bm{\beta}^{k+1}),\label{dual}
\end{align}
\end{subequations}
where $k$ is the iteration step, $\bm{p}$ is the dual variable introduced in the optimization problem and $\lambda>0$ is step size for dual variable update. \par
In the above iterative scheme, the optimal condition of (\ref{primal}) yields $\bm{p}^{k+1}\in\partial|\bm{\beta}^{k+1}|_1$ on combining with Equation (\ref{dual}). Equation (\ref{dual}) can be reformulated as
\begin{eqnarray}\label{frac}
\frac{\bm{p}^{k+1} - \bm{p}^k}{\lambda} = \Sigma^T(\bm{e}-\Sigma\bm{\beta}^{k+1}).
\end{eqnarray}

If we consider both primal variable $\bm{\beta}$ and dual variable $\bm{p}$ as functions of time $t$ and interpret $\lambda$ as the time step, BI can be transformed into a dynamic equation:
\begin{equation}\label{evo_eq1}
\begin{aligned}
&\partial_t\bm{p}(t)=\Sigma^T(\bm{e}-\Sigma\bm{\beta}(t)) \\
\text{ s.t }\;& \bm{p}(t)\in \partial |\beta(t)|_1.
\end{aligned}
\end{equation}

The above differential inclusion is called inverse~scale\ space\ flow (ISS)~\cite{burger2006nonlinear}. \par
The main idea of aISS \cite{burger2013adaptive} is that  the support of the solution $\bm{\beta}(t^k)$ can be determined by the subgradient $\bm{p}(t^{k})$ at some time $t^k$ based on the following relation
\begin{eqnarray}\label{evo_eq}
\bm{p}\in\partial|\bm{\beta}|_1 \Leftrightarrow
\left
\{\begin{array}{lll}
\bm{p}_i = \mathrm{sign}(\bm{\beta}_i),                       &\mathrm{if}\ \bm{\beta}_i\neq0, \\
  |\bm{p}_i|\leq1,                          &\mathrm{otherwise}.
\end{array}
\right.
\end{eqnarray}
for all $1\leq i\leq p$, i.e. the support of $\bm{\beta}(t^k)$ is restricted onto the set $I=\{i\mid|\bm{p}_i(t^k)|=1\}$. Furthermore, the step size $t^k$ can be determined as the change of the signal is piecewise linear, and  the signal was obtained  via solving a low dimensional optimization problem under the constraint of the sign consistence between $\bm{p}_i(t^k)$ and $\bm{\beta}_i(t^k)$.

\subsection{GISS$^{\rho}$ estimator}

Motivated by the similarities between greedy algorithms (e.g. OMP~\cite{pati1993orthogonal,mallat1993matching,tropp2007signal}, WOMP~\cite{tropp2004greed,foucart2012stability}, CoSaMP~\cite{needell2009cosamp} and HTP~\cite{foucart2011hard}) and $l_1$ minimization algorithms (e.g. aISS), the authors in \cite{Mller2016FastSR} proposed GISS method which approximates aISS flow closely. By removing the sign consistence constraint in aISS, GISS method achieves a more efficient scheme and provide a posterior method to estimate the distance to the $l_1$ minimizer. GISS method not only inherits remarkable properties from aISS, such as finite steps convergence but also shows much faster computation speed \cite{Mller2016FastSR}. The modified version,  namely GISS$^{\rho}$, can  further accelerate the computation speed  by including larger set of index with a factor $\rho\geq1$.

Because of the efficiency of GISS$^\rho$  algorithm for sparse recovery, we propose to apply it to estimate the  columns of precision matrix, based on the formulation (\ref{model1}). More specifically, we propose GISS$^{\rho}$ estimator by applying GISS$^{\rho}$ algorithm via solving the following minimization problem:
\begin{equation} \label{model3}
\begin{aligned}
    &\text{min}
    &&|\bm{\beta_i}|_1\\
    &\text{s.t.}
    &&\Sigma_n\bm{\beta_i} = \bm{e_i}, \; \bm{\beta_i}\in \mathbb{R}^{p}
\end{aligned}
\end{equation}
for all $1\leq i \leq p$ with the $l_\infty$ norm stopping criterion $i.e.$ $|\Sigma_n\bm{\beta_i} -\bm{e_i}|_\infty\leq\lambda_n$.  The detail for solving each subproblem  is shown in \textbf{Algorithm \ref{GISSrho}}.
\begin{algorithm}\label{al1}
\caption{GISS$^{\rho}$ method.}
\label{GISSrho}
\textbf{Parameters:} $\Sigma_n,\ \bm{e_i},~\rho\geq1,~\lambda_n\geq 0$ \\
\textbf{Initialization:} $\bm{r_0} = \bm{e_i}, t_1 = 1/|\Sigma_n^T\bm{r_0}|_\infty,\ \bm{p}(t_1)=t_1\Sigma_n^T\bm{r_0}$
\begin{algorithmic}
\While{$|\bm{r_k}|_\infty>\lambda_n$} \\
1. Compute $I_k=\{i\mid |\tilde{\bm{p}}_i(t_k)|=1\}$ \\
2. Compute $\bm{\beta}(t_k)=\mathrm{arg\ min}_{\bm{\beta}}\{|\Sigma_n{P_{I_k}}\bm{\beta}-\bm{e}|_2^2\}$ \\
3. Compute the residual $\bm{r_k}=\bm{e}-\Sigma_n\bm{\beta}(t^k)$ \\
4. Obtain $t_{k+1}$ as \par
(3.1)\ \ \  $t_{k+1}$ = $\rho$min\{$t\ |\ t>t_k,\exists j:|\bm{p}_j(t)|=1,\ \bm{\beta}_j(t_k)=0,\ \bm{p}_j(t)\neq \bm{p}_j(t_k)$\}, \\
where \par
(3.2)\ \ \  $ \bm{p}_j(t)=\bm{p}_j(t_k)+(t-t_k)(\Sigma_n^T\bm{r_k})_j$\\
5. Update the dual variable $\bm{p}(t)$ via (3.2) with $t=t_{k+1}$ and \par
$\tilde{\bm{p}}(t^{k+1})= $sign$(\bm{p}(t^{k+1}))$min$(|\bm{p}(t^{k+1})|,1)$
\EndWhile
\end{algorithmic}
\textbf{return} $\bm{\beta}(t_k)$ as $\bm{\beta_i}$.
\end{algorithm} \par
We note that as CLIME estimator, the estimation of each column can be performed in parallel. Also the difference between GISS$^{\rho}$ estimator and CLIME estimator is the role of $\lambda_n$. More precisely, $\lambda_n$ in CLIME estimator is used as the model tolerance. By controlling the tuning parameter, CLIME estimator can explore broader region of the solution space and the optimization problem is solved by a primal dual interior point algorithm applied on the equivalent linear programming problem. For  GISS$^{\rho}$ estimator,  $\lambda_n$ plays the role of the stopping criterion and the algorithm produces a sequence of estimators. As expected, from our observation, both in the theoretical analysis and simulation experiments, $\lambda_n$ shows similar results in GISS$^{\rho}$ estimator comparing to CLIME estimator, while  GISS$^{\rho}$ estimator is more efficient, especially for large scale problem. The detailed numerical performance will be shown in Section \ref{sec:ne}.
\section{Theoretical Analysis}
\label{sec:theo}
\subsection{Asymptotic analysis}
In this section, we will present the asymptotic results of GISS$^{\rho}$ estimator under standard moments conditions of the random variables. We denote the sample covariance matrix as $\textbf{$\Sigma_n$}=(\hat{\sigma}_{ij})=(\bm{\hat{\sigma}}_1,\bm{\hat{\sigma}}_2,\ldots,\bm{\hat{\sigma}}_p)$, where $\bm\hat{\sigma}_i,~i=1,\cdots, p$ are column vectors.
The true covariance matrix as $\textbf{$\Sigma_0$}=(\sigma_{ij}^0)$ and the expectation of a random variable (r.v.) $\bm{X}$ as E$\bm{X}=(\mu_1,\ldots,\mu_p)^T$. The following two  cases are usually considered according to the moment conditions on $\bm{X}$. \\

\noindent(C1) \textbf{Exponential-type tails}: Suppose that there exist some $0\textless\eta\textless1/4$ such that $\mathrm{log}p/n\leq\eta$ and
\begin{equation*}
  \mathrm{E}e^{t(X_i-\mu_i)^2}\leq K\textless\infty \quad \mathrm{for\ all\ } |t|\leq\eta,\ \mathrm{for\ all\ } i,
\end{equation*}
where $K$ is a upper bounded constant. \\
\noindent(C2) \textbf{Polynomial-type tails}: Suppose that for some $\tilde{\gamma},\ c_1\textgreater0,\ p\leq c_1 n^{\tilde{\gamma}},$ and for some $\delta\textgreater0$,
\begin{equation*}
  \mathrm{E}|X_i-\mu_i|^{4\tilde{\gamma}+4+\delta}\leq K\quad \mathrm{for\ all\ } i.
\end{equation*}\par

The following quantity is to measure the relation with the variance $\hat{\sigma}_{ij}$ which is useful in the theoretical property analysis,
\begin{equation}
  \kappa = \mathrm{max}_{ij}~\mathrm{E}[(X_i-\mu_i)(X_j-\mu_j)-\sigma_{ij}^0]^2=:\mathrm{max}_{ij}~\kappa_{ij}.
\end{equation}
The maximum value $\kappa$ captures the overall variability of $\Sigma_n$. With the condition (C1) and (C2), $\kappa$ is a bounded constant depending on $\tilde{\gamma},\delta,K$~\cite{CLIME}. \par
Through the following convergence analysis, we consider the uniformity class of matrices  to estimate the precision matrix $\Omega_0$,
\begin{equation*}
\begin{aligned}
  U:=U(q,s_0(p))=\{\Omega:\Omega\succ0,\; \|\Omega\|_{L_1}\leq M,\; \mathrm{max}_{1\leq i\leq p}\sum_{j=1}^p|\omega_{ij}|^q\leq s_0(p)  \}
\end{aligned}
\end{equation*}
for $0\leq q<1$, where $\Omega=:(\omega_{ij})=(\bm{\omega}_1,\bm{\omega}_2,\cdots,\bm{\omega}_p)$. For the special case $q=0$, $U(0,s_0(p))$ is a class of $s_0(p)$-sparse matrices.

\begin{theorem}\label{Convergence-Th}
  Suppose that $\Omega_0\in U(q,s_0(p))$ and $\hat{\Omega}$ is the solution solved through GISS$^{\rho}$ process with $\rho\geq1$.
  \begin{itemize}
   \item Assume that (C1) holds. Let $\lambda_n=C_0M\sqrt{\frac{\mathrm{log}p}{n}}\textless\frac{1}{p}-\varrho$ for some $\varrho\textgreater0$, where $C_0=2\eta^{-2}(2+\tau+\eta^{-1}e^2K^2)^2$ and $\tau>0$. Then
      \begin{align}
       &|\hat{\Omega}-\Omega_0|_\infty\leq C_0\left(3+\frac{1+3p\lambda_n}{1-p\lambda_n}\right)M^2\sqrt{\frac{\mathrm{log}p}{n}},\label{infty_ineq_assunp1}
      \end{align}
   with probability at least $1-4p^{-\tau}$.

   \item Assume that (C2) holds. Let $\lambda_n=D_0M\sqrt{\frac{\mathrm{log}p}{n}}\textless\frac{1}{p}-\varrho$ for some $\varrho\textgreater0$, where $D_0=\sqrt{(5+\tau)(\kappa+1)}$. Then
      \begin{align}
       &|\hat{\Omega}-\Omega_0|_\infty\leq D_0\left(3+\frac{1+3p\lambda_n}{1-p\lambda_n}\right)M^2\sqrt{\frac{\mathrm{log}p}{n}}\label{infty_ineq_assunp2},
    \end{align}
    with probability at least $1-O\left(n^{-\delta/8}+p^{-\tau/2}\right)$.
    \end{itemize}
\end{theorem}

From the Theorem \ref{Convergence-Th}, we find that the convergence rate of our estimator is  same to the one of CLIME estimator \cite{CLIME} in the sense of elementwise $l_\infty$ norm under two types of moment conditions, which consequently outperform $l_1$-MLE type estimators in the case of polynomial-type tails \cite{ravikumar2011high}.

Similarly, we can obtain the following convergence results:
\begin{theorem}\label{Convergence-Th(C3)}
  Suppose $\Omega_0\in U(q,s_0(p))$ and ($C1$) holds. Let $\lambda_n=C_0M\sqrt{\frac{\mathrm{log}p}{n}}\textless\frac{1}{p}-\varrho$ for some $\varrho\textgreater0$ with $C_0$ defined in Theorem \ref{Convergence-Th} and $\tau$ is sufficiently large. Let $\gamma=\sqrt{\frac{\mathrm{log}p}{n}}$. If $p\geq n^{\Xi}$ for some $\Xi\geq0$. Then
  \begin{align}
   &\mathrm{sup}_{\Omega_0\in U}\mathrm{E}|\hat{\Omega}_\gamma-\Omega_0|^2_\infty=O\left(M^4\frac{\mathrm{log}p}{n}\right),
 \end{align}
  \end{theorem}\par

 \subsection{Sparsity  recovery properties}
In this subsection, we provide some analysis from the point of view of sparsity recovery in compressive sensing \cite{tropp2004greed,zhao2006model,donoho2001uncertainty} as the original idea of GISS$^{\rho}$ algorithm  from the signal processing for the noisy free case.   Since our method can be computed in the parallelism form, we consider only one column recovery of the precision for simplicity. Denote $\bm{r}=\hat{\Sigma}\bm{\beta}^*-\bm{e}$ where $\bm{e}\in\mathbb{R}^p$ stands for one column of the identity matrix ${I}_p$, $\hat{\Sigma}$ is the $p\times p$ sample or true covariance matrix and $\bm{\beta}^*$ is the related column of the precision matrix to be recovered. In the algorithm, the algorithm is stopped with the criteria $|\bm{r}|_\infty\leq \lambda_n$.  In the above section, we have analyzed the convergence of the approximation in terms of $\lambda_n$. In this section, we will show two theoretical results on the sparsity recovery guarantee with the assumption that the residual is a gaussian random variable.  More precisely, we analyze the setting that the linear operator as the covariance matrix or the approximation of covariance matrix as a linear operator, the columns of the precision matrix is the sparse vector to be recovered. The problem is reformulated as a sparse  recovery problem as followed,
\begin{equation}\label{gaussian_model}
  \bm{r}= \hat{\Sigma}\bm{\beta}^*-\bm{e} \sim N(\bm{0},\epsilon^2I_p).
\end{equation}

The assumption of the residual or observation error  is to be Gaussian with standard variation $\epsilon$. We assume that each $\bm{\beta}^*$ has less than $s\leq p$ nonzero components.
For convenience, let $S$ = supp($\bm{\beta}^*$) be the support of the nonzero index  set of $\bm{\beta^*}$ and $T$ be its complement set. $\hat{\Sigma}_S$ denotes the submatrix of $\hat{\Sigma}$ formed by the columns of $\hat{\Sigma}$ in $S$, which are assumed to be linearly independent. One can similarity  define $\hat{\Sigma}_T$.\par
Before providing the sparsity recovery guarantee, we present necessary assumptions on the design matrix $i.e.$ covariance matrix $\hat{\Sigma}$.\\
(A1) $\textbf{Mutual\ Incoherence\ Condition}$:
\begin{equation*}
\quad  \mu:=\mathrm{max}_{i,j}\left|\frac{1}{p}\langle \hat{\Sigma}_i,\hat{\Sigma}_j\rangle\right|\textless\frac{1}{2s-1}, \quad s=|S|.
\end{equation*}
It can be shown \cite{tropp2004greed,burger2013adaptive} that once A1 holds, then
\begin{equation*}
  \theta=1-\mu(s-1)
\end{equation*}
and
\begin{equation*}
  \vartheta = \frac{1-\mu(2s-1)}{1-\mu(s-1)}.
\end{equation*}
where $\theta$ and $\vartheta$ satisfies $\hat{\Sigma}_{S}^*\hat{\Sigma}_{S}\geq \theta I$ and $\|\hat{\Sigma}_T^*\hat{\Sigma}_S^\dagger\|_\infty\leq 1-\vartheta$ for $\hat{\Sigma}_S^\dagger:=\hat{\Sigma}_S(\frac{1}{p}\hat{\Sigma}_S^T\hat{\Sigma}_S)^{-1}$ defined in Restricted Strong Convexity and  Irrepresentable Conditions respectively. Moreover,
 the above two popular conditions can be verified with condition A1, while they are more difficult to be checked in practice \cite{tropp2004greed,burger2013adaptive}.

In the following, we will show that  with large probability, the solution with the stopping rule  $|\bm{r}|_\infty\leq \lambda_n$ recovers  the true subset of the support index, under the assumption of Gaussian residual.

We further define the residual of the iterate $\bm{r}(t)=\bm{e}-\hat{\Sigma}\bm{\beta}(t)$ and then define the largest and the smallest nonzero magnitudes of $\bm{\beta}^*$ by $\bm{\beta}^*_\mathrm{max}:=\mathrm{max}(|\bm{\beta}_i^*|:i\in S)$ and $\bm{\beta}^*_\mathrm{min}:=\mathrm{min}(|\bm{\beta}_i^*|:i\in S)$ respectively.

The following Lemma show that the Gaussian noise is essentially bounded~\cite{cai2009recovery,cai2011orthogonal}.
\begin{lemma}
  For  $\bm{r}\sim N(\bm{0},\epsilon^2I_p)$ satisfies
  \begin{subequations}
  \begin{align}
    &P(\bm{r}\in B_2)\geq1-\frac{1}{p} \label{2-gau} \\
    &P(\bm{r}\in B_\infty(\varsigma))\geq1-\frac{1}{2p^\varsigma\sqrt{\pi\mathrm{log}p}} \label{infty-gau}
  \end{align}
  \end{subequations}
  where $B_2=\left\{\bm{r}:|\bm{r}|_2\leq\epsilon\sqrt{p+2\sqrt{p\mathrm{log}p}}\right\}$ and $B_\infty(\varsigma)=\left\{\bm{r}:|\hat{\Sigma}^T\bm{r}|_\infty\leq\epsilon\sqrt{2(1+\varsigma)\mathrm{log}p}\right\}$ with $\varsigma>0$.
\end{lemma}
For the two essentially bounded cases, we have
\begin{theorem}\label{Th1_gaussian}
Assume that (A1) holds and  $\bm{r}\sim N(\bm{0},\epsilon^2I_p)$, then
\begin{itemize}
\item For $l_2$ bounded Gaussian noise case (refer to (\ref{2-gau})),
   \begin{equation}
     \bm{\beta}^*_{\mathrm{min}}\geq\frac{2\epsilon}{\sqrt{\theta}}\left(\sqrt{1+2\sqrt{\frac{\mathrm{log} p}{p}}}+\sqrt{\frac{\mathrm{log}s}{p}}\right)
   \end{equation}
The GISS$^{\rho}$ algorithm with the stopping rule $\|\bm{r}(t)\|_\infty\leq\epsilon\sqrt{1+2\sqrt{\frac{\mathrm{log}p}{p}}}$ selects the true subset of $S$ with probability at least $1-\frac{\mathrm{exp}{\left(-\frac{p}{2}+\frac{1}{2}\left(1+2\sqrt{\frac{\mathrm{log}p}{p}}\right)+\frac{p}{2}\mathrm{log}\left(2-\frac{1}{p}\left(1+2\sqrt{\frac{\mathrm{log}p}{p}}\right)\right)\right)}}{1-\frac{1}{p}\left(1+2\sqrt{\frac{\mathrm{log}p}{p}}\right)}-O(p^{-1})$.
\item For  $l_\infty$ bounded Gaussian noise case (refer to (\ref{infty-gau})),
  \begin{equation}
    \bm{\beta}^*_{\mathrm{min}}\geq\frac{2\epsilon\sqrt{\mathrm{max}_i\|\hat{\Sigma}_i\|_p (1+\varsigma)s\mathrm{log}p}}{{p\theta}\|\hat{\Sigma}\|_{L_1}}+2\epsilon\sqrt{\frac{\mathrm{log}s}{p\theta}}.
  \end{equation}
The GISS$^{\rho}$ with the stopping rule $\|\bm{r}(t)\|_\infty\leq2\epsilon\|\hat{\Sigma}\|_{L_1}^{-1}{\sqrt{\mathrm{max}_i\|\hat{\Sigma}_i\|_p {(1+\varsigma)}\mathrm{log}p}}$ selects the true subset $S$ with probability at least  $1-O(p^{-1})$.
\end{itemize}
\end{theorem}
Theorem \ref{Th1_gaussian} introduces a stopping rule with $l_\infty$ norm which only depends the residual level assumption and the conditions on $\beta^*_\mathrm{min}$  to ensure that all significant variables are selected~\cite{cai2011orthogonal}.
\section{Numerical Experiments}
\label{sec:ne}
\subsection{Experiments Setup}
In the data simulation, we consider the following three cases.
\begin{itemize}
\item Case 1: Solve the sparse matrix with a designed covariance matrix: $\Sigma_{ij} = 0.5^{|i-j|}$. The exact precision matrix $\Omega=\Sigma^{-1}$ is given as
\begin{eqnarray}
\Omega_{ij}=\left\{\begin{array}{lll}
  1.3333,                      &\mathrm{for}~ i=j=1,p \\
  1.6667,                      &\mathrm{for}~ 1<i=j<p \\
-0.6667,                       &|i-j|=1 \\
  0,                           &\mathrm{otherwise}.
\end{array}\right.
\end{eqnarray}

\item Case 2: The second simulation refers to \cite{rothman2008sparse}. Assume that the true precision matrix $\Omega=B+\delta I$, where each off-diagonal entry in $B$ is generated independently and equals 0.5 with probability 0.1 or 0 with probability 0.9. $\delta$ is chosen such that conditional number of $\Omega$ is equal to $p$ and  the matrix is normalized to have the unit diagonals. Then we sample a set of data from the $p-$dimensional normal distribution $N(0,\Sigma)$ where $\Sigma=\Omega^{-1}$. The estimated covariance matrix from this set of generated data is used as $\Sigma_n$ for the recovery.
\item Case  3: Application to the fMRI data for ADHD. Attention Deficit Hyperactivity Disorder (ADHD) is a mental disorder causing 5\%-10\% of school-age children in behavior controlling through the United States. The \href{http://fcon1000.projects.nitrc.org/indi/adhd200/}{ADHD-200 project} released a resting-state fMRI dataset of healthy controls and ADHD children for scientific research. The data we processed is from Kennedy Krieger Institute (KKI), one of participating centers. The preprocessing pipeline can be found in \href{https://www.nitrc.org/plugins/mwiki/~index.php/neurobureau:AthenaPipeline}{https://www.nitrc.org/plugins/mwiki/~index.php/neurobureau:AthenaPipeline}. After preprocessing, we get 148 time points from each of 116 brain regions for each subject. We use the data of each subject to generate the sample covariance matrix for estimating the precision matrix which is similar to the methodology of numerical examples in \cite{liu2015fast}.\par

\end{itemize}

The above three simulation setting are tested with different algorithms for a comparison to ours.  The first one is Hard Thresholding Pursuit (HTP)  proposed in ~\cite{foucart2011hard} for solving the basis pursuit problem similar to (\ref{model3}). The algorithm is a combination of the Iterative Hard Thresholding algorithm and the Compressive Sampling Matching Pursuit algorithm (Cosamp). The other algorithms that we draw into comparisons are based on $l_1$ regularizer, such as L1-magic ~\cite{candes2005l1} and Alternating direction method of multipliers (ADMM) \cite{boyd2004convex,gabay1975dual,chan1978finite}. We denote $\text{ADMM}_{\lambda_n}$ for solving inequality constrained  model (\ref{model2}) and ADMM for the equality constrained model (\ref{model3}) for comparison.

The computation time will be reported with a  computation environment : MATLAB R2016b, Intel(R) Xeon(R) CPU E5-2689 v4 @3.10GHz.
\subsection{Results}
The result for Case 1 with different dimension $p$ is shown in the Table~\ref{Tabel1}. The third column of the table shows the non-zero entries number of the different methods with the thresholding value $10^{-8}$ applied on the recovered precision matrix for those $l_1$ based methods. The third column shows the result with the thresholding value $10^{-4}$, as we find all the methods successfully recover the true support. The last two columns show the computation time and the relative error to the true solution in terms of Frobenius norm.  We find that both our method GISS$^{\rho}$ (here $\rho=1$) have good computation time performance, especially in the high dimensional settings.

\begin{table}[h!]
\centering
\caption{Case 1}
\label{Tabel1}
\begin{tabular}{|l|l|l|l|l|l|}
\hline
\multirow{2}{*}{Dimension} & \multicolumn{1}{c|}{\multirow{2}{*}{\begin{tabular}[c]{@{}c@{}}Algorithms\end{tabular}}} &  \multicolumn{1}{c|}{\multirow{2}{*}{\begin{tabular}[c]{@{}c@{}}$L_0$ Norm\\ (1e-8)\end{tabular}}} & \multicolumn{1}{c|}{\multirow{2}{*}{\begin{tabular}[c]{@{}c@{}}$L_0$ Norm\\ (1e-4)\end{tabular}}} & \multicolumn{1}{c|}{\multirow{2}{*}{\begin{tabular}[c]{@{}c@{}}Time consuming\\ (s)\end{tabular}}} & \multicolumn{1}{c|}{\multirow{2}{*}{\begin{tabular}[c]{@{}c@{}}Relative Error\\ (Frobenius Norm)\end{tabular}}} \\
                           & \multicolumn{1}{c|}{}                                                                          &                   & \multicolumn{1}{c|}{}                                                                          & \multicolumn{1}{c|}{}                                                                              & \multicolumn{1}{c|}{}                                                                                           \\ \hline
\multirow{5}{*}{200}       & GISS                                                                                           & 598               & \multirow{5}{*}{598}                                                                           & 0.1084                                                                                             & {9.10e-10}                                                                                                        \\ \cline{2-3} \cline{5-6}
                           & HTP                                                                                           & 598               &                                                                                                & \textbf{0.0826}                                                                                    & {9.10e-10}                                                                                                        \\ \cline{2-3} \cline{5-6}
                           & L1-magic                                                                                         & 1780              &                                                                                                & 0.4679                                                                                            & 1.00e-09                                                                                                        \\ \cline{2-3} \cline{5-6}
                           & $\text{ADMM}_{\lambda_n}$                                                                                & 1388              &                                                                                                & 0.2649                                                                                             & 8.14e-10                                                                                                        \\ \cline{2-3} \cline{5-6}
                           & ADMM                                                                                           & 1390              &                                                                                                & 0.2615                                                                                             & 8.99e-10                                                                                                        \\ \hline
\multirow{5}{*}{400}       & GISS                                                                                           & 1198              & \multirow{4}{*}{1198}                                                                          & 0.1002                                                                                    & 9.09e-10                                                                                                        \\ \cline{2-3} \cline{5-6}
                           & HTP                                                                                           & 1198              &                                                                                                & \textbf{0.0975}                                                                                             & 9.09e-10                                                                                                        \\ \cline{2-3} \cline{5-6}
                           & L1-magic                                                                                         & 3580              &                                                                                                & 2.7838                                                                                            & 1.00e-09                                                                                                        \\ \cline{2-3} \cline{5-6}
                           & $\text{ADMM}_{\lambda_n}$                                                                                 & 2788              &                                                                                                & 1.5010                                                                                            & 8.05e-10                                                                                                        \\ \cline{2-3} \cline{5-6}
                           & ADMM                                                                                           & 2792              &                                                                                                & 1.5014                                                                                            & 8.98e-10                                                                                                        \\ \hline
\multirow{5}{*}{600}       & GISS                                                                                           & 1798              & \multirow{5}{*}{1798}                                                                          & \textbf{0.1479}                                                                                     & 9.09e-10                                                                                                        \\ \cline{2-3} \cline{5-6}
                           & HTP                                                                                           & 1798              &                                                                                                & 0.1567                                                                                             & 9.09e-10                                                                                                        \\ \cline{2-3} \cline{5-6}
                           & L1-magic                                                                                          & 5380              &                                                                                                & 16.4614                                                                                           & 1.04e-09                                                                                                        \\ \cline{2-3} \cline{5-6}
                           & $\text{ADMM}_{\lambda_n}$                                                                                & 4188              &                                                                                                & 9.7355                                                                                            & 8.02e-10                                                                                                        \\ \cline{2-3} \cline{5-6}
                           & ADMM                                                                                           & 4192              &                                                                                                & 10.4155                                                                                            & 8.97e-10                                                                                                        \\ \hline
\multirow{5}{*}{800}       & GISS                                                                                           & 2398              & \multirow{5}{*}{2398}                                                                          & \textbf{0.2518}                                                                                     & 9.09e-10                                                                                                        \\ \cline{2-3} \cline{5-6}
                           & HTP                                                                                           & 2398              &                                                                                                & 0.3942                                                                                            & 9.09e-10                                                                                                        \\ \cline{2-3} \cline{5-6}
                           & L1-magic                                                                                          & 7180              &                                                                                                & 113.4222                                                                                           & 1.04e-09                                                                                                        \\ \cline{2-3} \cline{5-6}
                           & $\text{ADMM}_{\lambda_n}$                                                                                 & 5588              &                                                                                                & 36.3055                                                                                            & 8.01e-10                                                                                                        \\ \cline{2-3} \cline{5-6}
                           & ADMM                                                                                           & 5592              &                                                                                                & 37.9462                                                                                            & 8.97e-10                                                                                                        \\ \hline
\multirow{5}{*}{1000}      & GISS                                                                                           & 2998              & \multirow{5}{*}{2998}                                                                          & \textbf{0.4105}                                                                                    & 9.09e-10                                                                                                        \\ \cline{2-3} \cline{5-6}
                           & HTP                                                                                           & 2998              &                                                                                                & 0.7366                                                                                             & 9.09e-10                                                                                                        \\ \cline{2-3} \cline{5-6}
                           & L1-magic                                                                                          & 8980              &                                                                                                & 392.0562                                                                                           & 1.04e-09                                                                                                        \\ \cline{2-3} \cline{5-6}
                           & $\text{ADMM}_{\lambda_n}$                                                                              & 6988              &                                                                                                & 74.3292                                                                                           & 8.00e-10                                                                                                        \\ \cline{2-3} \cline{5-6}
                           & ADMM                                                                                           & 6992              &                                                                                                & 77.3578                                                                                           & 8.97e-10                                                                                                        \\ \hline
\multirow{5}{*}{2000}      & GISS                                                                                           & 5998              & \multirow{5}{*}{5998}                                                                          & \textbf{2.4252}                                                                                    & 9.09e-10                                                                                                        \\ \cline{2-3} \cline{5-6}
                           & HTP                                                                                           & 5998              &                                                                                                & 5.4236                                                                                           & 9.09e-10                                                                                                        \\ \cline{2-3} \cline{5-6}
                           & L1-magic                                                                                          & 17980             &                                                                                                & 6.1685e+03                                                                                          & 1.04e-09                                                                                                        \\ \cline{2-3} \cline{5-6}
                           & $\text{ADMM}_{\lambda_n}$                                                                                  & 13988             &                                                                                                & 709.4917                                                                                          & 7.98e-10                                                                                                        \\ \cline{2-3} \cline{5-6}
                           & ADMM                                                                                           & 13992             &                                                                                                & 733.7321                                                                                          & 8.97e-10                                                                                                        \\ \hline
\end{tabular}
\end{table}

For Case 2, we vary the dimension  $p$ from $200$ to $600$ and replicate 20 times for each algorithms. We adopt 20 cores for a parallelization implementation for this case. In this numerical setting, we take $\rho= 1.2, 1.5, 2.0 > 1$ in GISS$^{\rho}$ for different $p$ and set 0.05 for the thresholding value. Table \ref{my-label2}, \ref{my-label3}, \ref{my-label4} and \ref{TP_TN} show the averages errors (Mean) and standard errors (SE) of losses in terms of different matrix norms: Frobenius norm, $L_1$ norm, $L_2$ norm and elementwise $l_\infty$ norm. True positive (TP) number and True negative (TN) number are also computed. The last column is the average time consuming of each algorithm. {For $p\geq n$ situation, we show the numerical results in Table \ref{p_large_n_1} and \ref{p_large_n_2} for $p=60,~100$ respectively. We set 0.01 for the thresholding value and we found the Frobenius norm results outperform other methods, although the SE are larger than others.} 
 \par
\begin{table}[]
\centering
\caption{Case 2, $p=200$, $\lambda_{200}$=$0.7\sqrt{\frac{\mathrm{log}p}{n}}$ with sample number $n$= 3000.}
\begin{tabular}{|l|l|l|l|l|l|l|}
\hline
$p$=200         & \multicolumn{1}{c|}{Frobenius norm} & Matrix $L_1$ norm              & Operator norm              & \multicolumn{1}{c|}{$l_\infty$ Norm}  & \multicolumn{1}{c|}{\begin{tabular}[c]{@{}c@{}}Time consuming\\ (s)\end{tabular}} \\ \hline
GISS            & 0.5457(0.008)                   & 0.1084(0.016)          & \textbf{0.0125(0.003)}   & 0.1314(0.003)                       & \textbf{0.1219}                                                                   \\ \hline
HTP            & \textbf{0.2485(0.007)}         & \textbf{0.0884(0.016)} & 0.0571(0.003)          & \textbf{0.1088(0.003)}              & 0.3885                                                                            \\ \hline
ADMM            & 0.3343(0.014)                   & 0.2107(0.035)         & 0.0793(0.008)          & 0.1412(0.003)                       & 0.1308                                                                            \\ \hline
ADMM$_{\lambda_n}$ & 1.2255(0.009)                  & 0.5865(0.038)          & 0.2515(0.007)         & 0.1530(0.003)                       & 0.4589                                                                            \\ \hline
\end{tabular}
\label{my-label2}
\end{table}
\begin{table}[]
\centering
\caption{Case 2, $p=400$, $\lambda_{400}$=$0.7\sqrt{\frac{\mathrm{log}p}{n}}$ with sample number $n$= 5000.}
\begin{tabular}{|l|l|l|l|l|l|l|}
\hline
$p$=400         & \multicolumn{1}{c|}{Frobenius norm} & Matrix $L_1$ norm              & Operator norm        & \multicolumn{1}{c|}{$l_\infty$ Norm}  & \multicolumn{1}{c|}{\begin{tabular}[c]{@{}c@{}}Time consuming\\ (s)\end{tabular}} \\ \hline
GISS            & 1.4850(0.011)                   & 0.2796(0.028)          & 0.1253(0.004) & 0.1478(0.002)                       & \textbf{0.8004}                                                                   \\ \hline
HTP             & \textbf{0.4452(0.009)}          & \textbf{0.1192(0.018)} & \textbf{0.1187(0.003)}          & \textbf{0.1056(0.003)}              & 6.3441                                                                            \\ \hline
ADMM            & 0.9107(0.023)                   & 0.1878(0.025)          & 0.1896(0.007)          & 0.1940(0.004)                       & 0.8222                                                                            \\ \hline
ADMM$_{\lambda_n}$ & 2.0134(0.012)                   & 0.5611(0.030)          & 0.3289(0.005)          & 0.1606(0.002)                       & 8.4315                                                                            \\ \hline
\end{tabular}
\label{my-label3}
\end{table}
\begin{table}[]
\centering
\caption{Case 2, $p=600$, $\lambda_{600}$=$0.7\sqrt{\frac{\mathrm{log}p}{n}}$ with sample number $n$= 8000.}
\begin{tabular}{|l|l|l|l|l|l|l|}
\hline
$p$=600         & \multicolumn{1}{c|}{Frobenius norm} & Matrix $L_1$ norm              & Operator norm              & \multicolumn{1}{c|}{$l_\infty$ Norm} & \multicolumn{1}{c|}{\begin{tabular}[c]{@{}c@{}}Time consuming\\ (s)\end{tabular}} \\ \hline
GISS            & 1.6309(0.010)                   & \textbf{0.1536(0.022)} & \textbf{0.0171(0.003)} & 0.1290(0.001)                & \textbf{3.4359}                                                                   \\ \hline
HTP             & \textbf{0.4823(0.007)}          & 0.2450(0.027)         & 0.2918(0.003)          & \textbf{0.0845(0.001)}                & 43.1578                                                                           \\ \hline
ADMM            & 2.1578(0.023)                   & 0.3680(0.031)          & 0.0323(0.007)          & 0.1759(0.001)                       & 12.7318                                                                           \\ \hline
ADMM$_{\lambda_n}$ & 2.2408(0.009)                   & 0.5123(0.020)          & 0.2902(0.004)          & 0.1341(0.001)                      & 13.6197                                                                           \\ \hline
\end{tabular}
\label{my-label4}
\end{table}
\begin{table}[]
\centering
\caption{Case 2, $p=200,~400,~600$.}
\begin{tabular}{|l|l|l|l|l|l|l|}
\hline
\multirow{2}{*}{} & \multicolumn{3}{c|}{\textbf{TP(\%)}}             & \multicolumn{3}{c|}{\textbf{TN(\%)}}           \\ \cline{2-7}
                  & {200}   & 400   & {600}   & {200}   & {400}   & {600} \\ \hline
{GISS}     & 99.93          & 98.43          & 94.29          & 99.71          & 99.93          & \textbf{100} \\ \hline
{HTP}      & 99.95          & 97.95          & 92.35          & \textbf{99.81} & \textbf{99.99} & \textbf{100} \\ \hline
{ADMM}     & 99.75          & 96.01          & 91.89          & 99.1           & 99.94          & 99.97        \\ \hline
ADMM$_{\lambda_n}$     & \textbf{99.97} & \textbf{99.09} & \textbf{96.52} & 98.69          & 99.84          & 99.99        \\ \hline
\end{tabular}
\label{TP_TN}
\end{table}
\begin{table}[]
\centering
\caption{Case 2, $p=60$, $\lambda_{60}$=$0.001\sqrt{\frac{\mathrm{log}p}{n}}$ with sample number $n$= 30.}
\begin{tabular}{|l|l|l|l|l|}
\hline
$p$=60      & Frobenius norm & Matrix $L_1$ norm & Operator norm & $l_\infty$ Norm \\ \hline
GISS      & \textbf{1.3570(0.14)}   & 3.8780(0.21)  & 1.1073(0.14)  & 1.4990(0.15)            \\ \hline
HTP      & 8.4734(0.01)   & \textbf{2.8927(0.01)}  & 2.2451(0.01)  & \textbf{0.8689(0.01)}            \\ \hline
ADMM & 6.8352(0.00)   & 5.8557(0.02)  & \textbf{0.4718(0.00)}  & 1.0770(0.01)            \\ \hline
ADMM$_{\lambda_n}$  & 6.8379(0.00)   & 5.8608(0.02)  & 0.4719(0.00)    & 1.0770(0.01)            \\ \hline
\end{tabular}
\label{p_large_n_1}
\end{table}
\begin{table}[]
\centering
\caption{Case 2, $p=100$, $\lambda_{100}$=$0.001\sqrt{\frac{\mathrm{log}p}{n}}$ with sample number $n$= 50.}
\begin{tabular}{|l|l|l|l|l|}
\hline
$p$=100     & Frobenius norm & Matrix $L_1$ norm & Operator norm & $l_\infty$ Norm \\ \hline
GISS      & \textbf{1.6411(0.14) }  & 5.1000(0.22)  & 0.9443(0.10)  & 1.4988(0.12)            \\ \hline
HTP      & 11.2055(0.01)  & \textbf{4.1144(0.01) } & 2.7254(0.01)  & \textbf{0.9448(0.01) }           \\ \hline
ADMM & 12.7463(0.00)  & 9.2365(0.02)  & 0.7048(0.00)  & 1.4200(0.01)            \\ \hline
ADMM$_{\lambda_n}$  & 12.7505(0.00)  & 9.2454(0.02)  & \textbf{0.7048(0.00) } & 1.4200(0.01)            \\ \hline
\end{tabular}
\label{p_large_n_2}
\end{table}
\par
In Theorem \ref{Convergence-Th}, we showed the convergence rate of the estimator with respect to the elementwise $l_\infty$ norm.
From Figure \ref{Error_fig}, we can also numerically observe the linear convergence rate of $L_2$ norm and Frobenius norm, although this can not be shown theoretically in the current framework.
\begin{figure}[htbp]
	\centering	\subfigure{\includegraphics[width=1\textwidth]{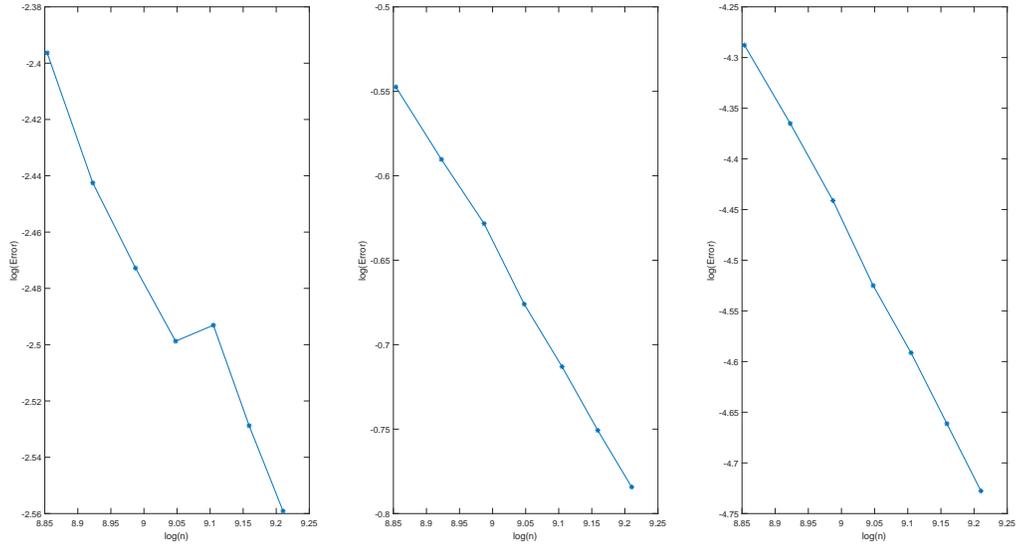}}
	\caption{Error estimated through GISS$^{\rho}$ algorithm (Left: Matrix elementwise $l_\infty$ norm; Middle: Operator norm; Right: Frobenius norm). We check the case in Model 2 when $p=400,~\rho=1$ and repeat the test for 20 times. The stopping criterion $\lambda_n$ is set as 0.05 and the sample number varies from 7000 to 10000. }
\label{Error_fig}
\end{figure}


 For Case 3 where the real fMRI dataset for ADHD is used, we compare the performance of support recovery using the data from all subjects, and the results suggest that GISS$^{\rho}$ has competitive performance with CLIME.
 Figure \ref{ADHD_GISS_fig} compares the running time of GISS$^{\rho}$ with different $\rho$, HTP algorithm, and $\text{ADMM}_{\lambda_n}$ algorithm. Similar to the procedure described before, for each subject, we use the each algorithm to recover a designated edge percentages: 10\%, 20\%, 30\%, 40\%, 50\%. This plot shows that the running time of GISS$^{\rho}$ grows quickly for $\rho=1$ when we need to recover more connections. However, when we change the value by taking $\rho$=2, we find both HTP and GISS$^{\rho}$ ($\rho=2$) shows the similar performance. $\text{ADMM}_{\lambda_n}$ algorithm is the slowest among above three methods. We note that we omit the application of ADMM method to this dataset as the computation time performance is much worse than any other algorithms in Figure \ref{ADHD_GISS_fig}.
\begin{figure}[htbp]
  \centering
   \subfigure{\includegraphics[width=0.85\textwidth]{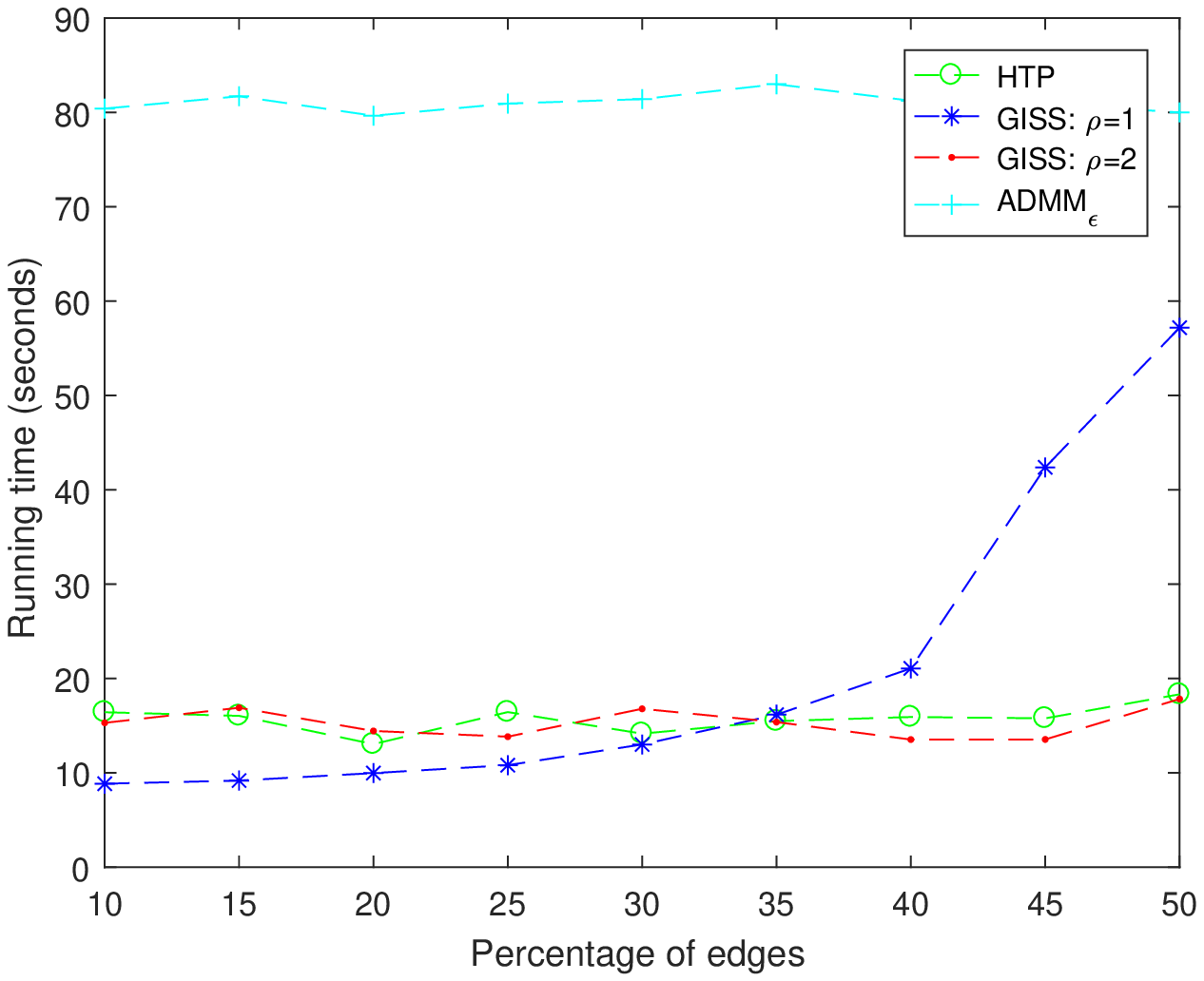}}
  \caption{Comparison of running time for the ADHD dataset.}
   \label{ADHD_GISS_fig}
\end{figure}
\section{Discussion and Conclusion}
\label{sec:dc}
This paper proposed a fast sparse algorithm GISS$^{\rho}$ for estimating precision matrix, especially for the high dimensional problem setting. The GISS$^{\rho}$ is a algorithm based on the $l_1$ minimization. However, it combines the advantage of the greedy algorithms. The asymptotic convergence and the sparse reovery properties of this algorithm are analyzed. Numerical simulations show that the proposed GISS$^{\rho}$ estimator shows a good performance compared to HTP, ADMM and other $l_1$ based estimators, especially for high dimensional problems.

\section*{Appendix}

\begin{lemma}\label{lem: 1}
  Suppose that $\Omega_0\in U(q,s_0(p))$ and $\hat{\Omega}_{n,\gamma}$ is the solution solved by GISS$^{\rho}$ for  $\rho\geq1$ and $\hat{\Sigma}_{n,\gamma}$  defined  as (\ref{eq:sigma_regu}) and symmetry operation defined in (\ref{sym}). If $\|\Omega_0\|_{L_1}(\mathrm{max}_{ij}|\hat{\sigma}_{ij}-\sigma_{ij}^0|+\gamma)\leq\lambda_n\leq\frac{1}{p}$, then we have
  \begin{align}
    &|\hat{\Omega}_{n, \gamma}-\Omega_0|_\infty\leq\left(3+{\frac{1+3p\lambda_n}{1-p\lambda_n}}\right)\lambda_nM,\label{infty_ineq}
  \end{align}
  where $C_1\leq2(1+2^{1-q}+3^{1-q})\left(3M+{\frac{1+3p\lambda_n}{1-p\lambda_n}M}\right)^{1-q}$.
\end{lemma}

\begin{proof}
According to the proof of Theorem 6 in \cite{CLIME}, we denote $\hat{\Omega}_{n,\gamma}:=\hat{\Omega}_{\gamma}$. Let $\hat{\bm{\beta}}$ be the one column of $\Omega_0$ and the $\bm{\beta}$ be the final result of GISS$^{\rho}$ estimator satisfying $|\Sigma_{n,\gamma}\bm{\beta}-\bm{e}|_\infty\leq\lambda_n$, then we have
  \begin{subequations}
  \begin{align}
    |\Sigma_0(\hat{\bm{\beta}}-\bm{\beta})|_\infty&\leq|\Sigma_{n,\gamma}(\hat{\bm{\beta}}-\bm{\beta})|_\infty+|(\Sigma_{n,\gamma}-\Sigma_0)(\hat{\bm{\beta}}-\bm{\beta})|_\infty \\
    &\leq|\Sigma_{n,\gamma}\hat{\bm{\beta}}-e|_\infty+|\Sigma_{n,\gamma}\bm{\beta}-e|_\infty+|\Sigma_{n,\gamma}-\Sigma_0|_\infty|\hat{\bm{\beta}}-\bm{\beta}|_1 \\
    &\leq|\Sigma_{n,\gamma}\hat{\bm{\beta}}-e|_\infty+\lambda_n+|\Sigma_{n,\gamma}-\Sigma_0|_\infty|\hat{\bm{\beta}}-\bm{\beta}|_1\\
    &\leq|(\Sigma_{n,\gamma}-\Sigma_0)\hat{\bm{\beta}}|_\infty+\lambda_n+|\Sigma_{n,\gamma}-\Sigma_0|_\infty|\hat{\bm{\beta}}-\bm{\beta}|_1 \\
    &\leq|\Sigma_{n,\gamma}-\Sigma_0|_\infty|\hat{\bm{\beta}}|_1+\lambda_n+|\Sigma_{n,\gamma}-\Sigma_0|_\infty|\hat{\bm{\beta}}-\bm{\beta}|_1 \\
    &\leq2\lambda_n+|\Sigma_{n,\gamma}-\Sigma_0|_\infty|\hat{\bm{\beta}}-\bm{\beta}|_1 \\
    &\leq2\lambda_n+|\Sigma_{n,\gamma}-\Sigma_0|_\infty(|\hat{\bm{\beta}}|_1+|\bm{\beta}|_1)\\
    &\leq3\lambda_n+\frac{\lambda_n|\bm{\beta}|_1}{\|\Omega_0\|_{L_1}}\label{res}
    \end{align}
  \end{subequations}
Transform the result (\ref{res}) into matrix form and let $\hat{\Omega}^1_\gamma$ be the matrix form of $\beta$, it follows that
\begin{equation}
  |\hat{\Omega}^1_\gamma-\Omega_0|_\infty\leq\|\Omega_0\|_{L_1}|\Sigma_0(\hat{\Omega}^1_\gamma-\Omega_0)|_\infty\leq3\lambda_n\|\Omega_0\|_{L_1}+{\lambda_n\|\hat{\Omega}^1_\gamma\|_{L_1}}, \label{Lemma1_im}
\end{equation}
since
\begin{equation}
  |AB|_\infty\leq\|A\|_{L_1}|B|_\infty,
\end{equation}
when $A$ is symmetric matrix. \\
\textcolor{black}{Let $\tilde{\bm{p}}$ be the subgradient of $\|\cdot\|_1$ at $\bm{\beta}$, then
\begin{subequations}
\begin{align}
  |\langle\tilde{\bm{p}},\hat{\bm{\beta}}-\bm{\beta}\rangle|&\leq|\tilde{\bm{p}}|_1|\hat{\bm{\beta}}-\bm{\beta}|_\infty \\
  &\leq p|\Omega_0\Sigma_0(\hat{\bm{\beta}}-\bm{\beta})|_\infty\\
  &\leq p\|\Omega_0\|_{\infty}|\Sigma_0(\hat{\bm{\beta}}-\bm{\beta})|_\infty  \\
  &\textcolor{black}{=p\|\Omega_0\|_{L_1}|\Sigma_0(\hat{\bm{\beta}}-\bm{\beta})|_\infty}  \\
  &\leq p\|\Omega_0\|_{L_1}\left(3\lambda_n+\frac{\lambda_n|\bm{\beta}|_1}{\|\Omega_0\|_{L_1}}\right)
  \end{align}
\end{subequations}
and since
\begin{equation}
  |\hat{\bm{\beta}}|_1-|\bm{\beta}|_1-\langle\tilde{\bm{p}},\hat{\bm{\beta}}-\bm{\beta}\rangle\geq0,
\end{equation}
we have
\begin{subequations}
\begin{align}
  |{\bm{\beta}}|_1 &\leq |\hat{\bm{\beta}}|_1-\langle\tilde{p},\hat{\bm{\beta}}-\bm{\beta}\rangle\\
  &\leq|\hat{\bm{\beta}}|_1+p\|\Omega_0\|_{L_1}\left(3\lambda_n+\frac{\lambda_n|\bm{\beta}|_1}{\|\Omega_0\|_{L_1}}\right)
  \end{align}
\end{subequations}
In the matrix form, we have
\begin{equation}
  \textcolor{black}{\|\hat{\Omega}^1_\gamma\|_{L_1}\leq\|\Omega_0\|_{L_1}+3p\|\Omega_0\|_{L_1}\lambda_n+\lambda_n p\|\hat{\Omega}^1_\gamma\|_{L_1}}.
\end{equation}
Therefore,
\begin{equation}
  \textcolor{black}{\|\hat{\Omega}^1_\gamma\|_{L_1}\leq\left(\frac{1+3p\lambda_n}{1-p\lambda_n}\right)\|\Omega_0\|_{L_1} \label{relation_L1}}
\end{equation}
Combining (\ref{relation_L1}) and (\ref{Lemma1_im}), we have
\begin{equation}
  \textcolor{black}{|\hat{\Omega}^1_\gamma-\Omega_0|_\infty\leq\|\Omega_0\|_{L_1}|\Sigma_0(\hat{\Omega}^1_\gamma-\Omega_0)|_\infty\leq3\lambda_n\|\Omega_0\|_{L_1}+{\lambda_n\left(\frac{1+3p\lambda_n}{1-p\lambda_n}\right)\|\Omega_0\|_{L_1}}}. \label{Lemma1_im2}
\end{equation}
Following the symmetric operation of $\|\hat{\Omega}_\gamma\|_{L_1}$ introduced in (\ref{sym}), we have
\begin{equation}
  \textcolor{black}{|\hat{\Omega}_\gamma-\Omega_0|_\infty\leq|\hat{\Omega}^1_\gamma-\Omega_0|_\infty}.
\end{equation}
This establishes (\ref{infty_ineq}) with (\ref{Lemma1_im2})}.
\end{proof}

Combining the proof of Theorem 1(a) and 1(b) and Theorem 4(a), 4(b) in \cite{CLIME}, Theorem \ref{Convergence-Th} can be derived  with Lemma \ref{lem: 1}. Similarly, Theorem \ref{Convergence-Th(C3)} can be derived based on Theorem 2 and Theorem 5 in \cite{CLIME} and Lemma \ref{lem: 1}.

In the following, we prove the sparse recovery property Theorem  \ref{Th1_gaussian}.
\begin{proof}
We first prove the first statement in Theorem \ref{Th1_gaussian}.  Following the proof of Theorem 6.1 and 6.2 in \cite{osher2016sparse}, the proof  is divided into two parts. The first part shows that GISS$^{\rho}$ will stop at the chosen stopping criteria with high probability and the second part is to assure the process will not stop if there still have supports to choose.

  Let $b_\infty$ be the stopping point in the form of $l_{\infty}$ norm i.e. $b_\infty=\epsilon\sqrt{1+2\sqrt{\frac{\mathrm{log}p}{p}}}$. ${P}_t=\hat{\Sigma}_{S(t)}\left(\hat{\Sigma}^*_{S(t)}\hat{\Sigma}_{S(t)}\right)^{-1}\hat{\Sigma}_{S(t)}^*$. Obviously, $P_t$ is a symmetric projection matrix. It follows from Lemma 4 in \cite{cai2002block}, we have
    \begin{subequations}
    \begin{align}
      \mathrm{Prob}(|\bm{r}(t)|_\infty&=|(\textbf{I}-{P}_t)\bm{\epsilon}|_\infty\geq b_\infty)\leq\mathrm{Prob}(|\bm{\epsilon}|_2\geq b_\infty) \\
      &\leq\frac{1}{1-\frac{1}{p}\left(1+2\sqrt{\frac{\mathrm{log}p}{p}}\right)}\mathrm{exp}\left(-\frac{p}{2}\left(1-\frac{1}{p}\left(1+2\sqrt{\frac{\mathrm{log}p}{p}}\right)
      -\mathrm{log}\left(2-\frac{1}{p}\left(1+2\sqrt{\frac{\mathrm{log}p}{p}}\right)\right)\right)\right) \\
      &=\frac{\mathrm{exp}{\left(-\frac{p}{2}+\frac{1}{2}\left(1+2\sqrt{\frac{\mathrm{log}p}{p}}\right)+\frac{p}{2}\mathrm{log}\left(2-\frac{1}{p}\left(1+2\sqrt{\frac{\mathrm{log}p}{p}}\right)\right)\right)}}{1-\frac{1}{p}\left(1+2\sqrt{\frac{\mathrm{log}p}{p}}\right)}
    \end{align}
    \end{subequations}
    which means the algorithm stops when the residual is less than $b_\infty$. \\
    The following part is to show that the algorithm will not stop whenever there exits the $i\in S$ such that $\bm{\beta}_i^*(t)=0$. Similar to the proof of the Theorem 6.1 in \cite{osher2016sparse}, we have
        \begin{subequations}
    \begin{eqnarray}
    |\bm{r}(t)|_\infty &=& |\hat{\Sigma}_S(\bm{\bm{\beta}}^*-\bm{\beta}_S(t))+(I-P_S)\epsilon|_\infty \\
    &\geq&|\hat{\Sigma}_S(\bm{\beta}^*-\bm{\beta}_S(t))|_\infty-|(I-P_S)\epsilon|_\infty \\
    &\geq&\sqrt{\theta}|\bm{\beta}^*-\bm{\beta}_S(t)|_2-b_\infty \\
    &\geq&\sqrt{\theta}\bm{\beta}^*_{\mathrm{min}}-b_\infty \\
    &\geq&b_\infty
    \end{eqnarray}
        \end{subequations}
    provided that ${\bm{\beta}}^*_{\mathrm{min}}\geq\frac{2\epsilon}{\sqrt{\theta}}\left(\sqrt{1+2\sqrt{\frac{\mathrm{log}p}{p}}}\right)$. Since
    \begin{equation}
      |(\hat{\Sigma}_S^*\hat{\Sigma}_S)^{-1}\hat{\Sigma}_S^*\epsilon|_\infty\leq2\epsilon\sqrt{\frac{\mathrm{log}s}{p\theta}},\ \ \mathrm{w}.\mathrm{p}.~\ \mathrm{at}\ \mathrm{least}\ 1-2p^{-1},
    \end{equation}
    so it suffices to have
    \begin{equation}
     \bm{\beta}^*_{\mathrm{min}}\geq\frac{2\epsilon}{\sqrt{\theta}}\left(\sqrt{1+2\sqrt{\frac{\mathrm{log}p}{p}}}+\sqrt{\frac{\mathrm{log}s}{p}}\right).
   \end{equation}
\end{proof}

The second statement of  Theorem \ref{Th1_gaussian} can be similarly proved.

%

\end{document}